\documentclass[runningheads]{llncs}
\usepackage{graphicx}
\usepackage{epsfig}
\usepackage{graphics}
\usepackage{array}
\usepackage{amsmath}
\usepackage{algorithm}
\usepackage{algorithmic}
\usepackage{comment}
\usepackage{tikz}
\usepackage{epstopdf}
\usepackage{fullpage}
\date{}

\title{Tri-connectivity Augmentation in Trees}
\author{S.Dhanalakshmi, N.Sadagopan, D.Sunil Kumar} 
\institute{Indian Institute of Information Technology, Design and Manufacturing, Kancheepuram, Chennai, India. \\
\email{\{mat12d001,sadagopan\}@iiitdm.ac.in}}
\begin{document}
\maketitle
\begin{abstract}
For a connected graph, a {\em minimum vertex separator} is a minimum set of vertices whose removal creates at least two connected components.  The vertex connectivity of the graph refers to the size of the minimum vertex separator and a graph is $k$-vertex connected if its vertex connectivity is $k$, $k\geq 1$.  Given a $k$-vertex connected graph $G$, the combinatorial problem {\em vertex connectivity augmentation} asks for a minimum number of edges whose augmentation to $G$ makes the resulting graph $(k+1)$-vertex connected.  In this paper, we initiate the study of $r$-vertex connectivity augmentation whose objective is to find a $(k+r)$-vertex connected graph by augmenting a minimum number of edges to a $k$-vertex connected graph, $r \geq 1$.  We shall investigate this question for the special case when $G$ is a tree and $r=2$.  In particular, we present a polynomial-time algorithm to find a minimum set of edges whose augmentation to a tree makes it 3-vertex connected.  Using lower bound arguments, we show that any tri-vertex connectivity augmentation of trees requires at least $\lceil \frac {2l_1+l_2}{2} \rceil$ edges, where $l_1$ and $l_2$ denote the number of degree one vertices and degree two vertices, respectively. Further, we establish that our algorithm indeed augments this number, thus yielding an optimum algorithm.
\end{abstract}
\section{Introduction}
The study of vertex separators and associated combinatorial problems has been a fascinating research in the field of combinatorial computing.  One such classical problem, namely {\em connectivity augmentation} focuses on increasing the vertex connectivity by one by augmenting a minimum number of edges.  This study was initiated by Eswaran et al  \cite{tarjan} and they studied the fundamental problem {\em bi-connectivity augmentation}: given a 1-connected graph $G$, find a minimum number of edges whose augmentation to $G$ makes it 2-vertex connected (bi-vertex connected).  Subsequently, Hsu \cite{Watanabe,tri} studied the tri-connectivity augmentation of bi-vertex connected graphs, making a bi-vertex connected graph 3-vertex connected by augmenting a minimum number of edges.  The complexity of general vertex connectivity augmentation, i.e., given a $k$-vertex connected graph, find a minimum number of edges whose augmentation to the given graph makes it $(k+1)$-vertex connected, was settled by Vegh \cite{vegh} and this result has been a breakthrough result as the complexity of which was open for al most three decades.  \\
There are some note-worthy results as far as bi (tri)-connectivity augmentation problems are concerned.  To solve bi-connectivity augmentation, in \cite{vijaya,tarjan}, the given 1-vertex connected graph is transformed into a {\em block tree} and the augmentation is done with the help of the block tree.  In \cite{nsn}, a bi-connected component tree transforms the 1-connected graph and helps in obtaining a biconnectivity augmentation set.  Interestingly, in both the approaches a minimum bi-vertex connectivity augmentation can be obtained with the help of proposed tree-like graphs.   Recent work due to Surabhi et. al \cite{surabi} developed a strategy using which one can augment many edges in parallel, thus obtaining a simpler approach for sequential and parallel bi-vertex connectivity augmentation.  However, their work is restricted to the class of trees which are 1-vertex connected.  Similarly, for tri-connectivity augmentation of 2-vertex connected graphs, a 3-block tree of 2-vertex connected graphs was used to obtain a minimum tri-connectivity augmentation set \cite{vijaya,nsn}. \\
In this article, we initiate the study of $r$-vertex connectivity augmentation which is to find a $(k+r)$-vertex connected graph from a $k$-vertex connected graph by augmenting a minimum number of edges.  Having known the results of \cite{vegh}, it is natural to know whether iterative application of algorithm mentioned in \cite{vegh} $r$ times will yield a $(k+r)$-vertex connected graph.  It is important to highlight the fact that even for tri-connectivity augmentation of 1-connected graphs the chain approach fails to produce a minimum connectivity augmentation set, i.e, the approach of making a 1-vertex connected graph 2-vertex connected and making a 2-vertex connected graph 3-vertex connected does not yield an optimum augmentation set.  This calls for a good understanding of minimum vertex separators of $k$-vertex connected graphs and their role in $r$-vertex connectivity augmentation.  Towards this attempt, we shall explore the study of tri-connectivity augmentation in trees.  In particular, we present the following results in this paper;
\begin{itemize}
\item Given a tree $T$, any minimum tri-connectivity augmentation set has at least $\lceil \frac {2 l_1 + l_2}{2} \rceil$ edges, where $l_1$ and $l_2$ denote the number of degree one vertices and degree two vertices, respectively.
\item A polynomial-time algorithm to compute a minimum tri-connectivity augmentation set meeting the above bound.
\end{itemize}
We believe that the results presented in this paper can be extended to tri-connectivity augmentation of 1-connected graphs.\\
{\bf Roadmap:} In the next section, we present lower bound results for tri-connectivity augmentation followed by an algorithm which will yield the minimum tri-connectivity augmentation set in polynomial time.  We conclude this paper with some directions for $r$-connectivity augmentation of 1-connected graphs, $r \geq 3$. 
\subsection{Connectivity Augmentation Preliminaries}
Notation and definitions are as per \cite{West,Golumbic,cormen}.  Let $G =(V,E)$ be an undirected connected graph where $V(G)$ is the set of vertices and $E(G) \subseteq \{\{u,v\}~|~ u,v \in V(G)$, $u \not= v \}$. For $v \in V(G)$, $N_G(v)=\{u~|~ \{u,v\} \in E(G)\}$ and $d_G(v)=|N_G(v)|$ refers to the degree of $v$ in $G$. Let $v \in V(G)$ is said to be a \emph{leaf} if $d_G(v) = 1$. $\delta(G)$ and $\Delta(G)$ refers to the minimum and maximum degree of $G$, respectively.  For simplicity, we use $\delta$ and $\Delta$ when the associated graph is clear from the context. $P_{uv} = (u=u_1, u_2, \ldots, u_k=v)$ is a \emph{path} defined on $V(P_{uv})=\{u=u_1, u_2, \ldots, u_k=v\}$ such that $E(P_{uv}) = \{\{u_i, u_{i+1}\}\vert \{u_i, u_{i+1}\} \in E(G), 1 \leq i \leq k-1\}$. For a graph $G$, we define $D_1 = \{v \mid d_G(v) = 1\}$ such that $l_1 = \vert D_1\vert$ refers to the number of vertices in $D_1$ and we define $D_2 = \{v \mid d_G(v) = 2\}$ such that $l_2 = \vert D_2\vert$ refers to the number of vertices in $D_2$. For $S \subset V(G)$, $G[S]$ denotes the graph induced on the set $S$ and $G \setminus S$ is the induced graph on the vertex set $V(G) \setminus S$.  A \emph{vertex separator} of a graph $G$ is a set $S \subseteq V(G)$ such that $G \setminus S$ has more than one connected component.  A vertex separator $S$ is said to be \emph{minimal} if there no proper subset $S'$ of $S$ such that $S'$ is a vertex separator. A \emph{minimum vertex separator} $S$ is a vertex separator  of least size and the cardinality of such $S$ is the vertex connectivity of a graph $G$, written  $\kappa(G)$.  A graph is $k$-\emph{vertex connected} if $\kappa(G)=k$.  If $\kappa(G)=1$ then the graph is 1-connected (also known as singly connected) and in such a graph a minimum vertex separator $S$ is a singleton set and the vertex $v \in S$ is a {\em cut-vertex} of $G$.  A \emph{cycle} is a connected graph in which the degree of each vertex is two.  A \emph{tree} is a connected and an acyclic graph.  For a graph $G$ with $\kappa(G)=k$, a minimum connectivity augmentation set $E_{ca} = \{\{u,v\} ~|~ u,v \in V(G) \mbox{ and } \{u,v\} \notin E(G) \}$ is such that the graph obtained from $G$ by augmenting $E_{ca}$ is of vertex connectivity $k+r$, $r \geq 2$. This paper is written in the context of augmenting $E_{ca}$ edges to a tree such that the obtained graph is $3$-connected.  
\section{Tri-connectivity Augmentation in Trees}
In this section, we shall first present the lower bound analysis which is a number representing the number of edges to be augmented in any minimum connectivity augmentation set to make a tree 3-vertex connected.  In the subsequent sections, we first give an sketch of the algorithm and then we shall present an algorithm with analysis which will output a connectivity augmentation set meeting the lower bound. Our approach finds a minimum tri-connectivity augmentation set for trees with $\Delta(G) \leq 2$ (which are called paths) and for trees with $\Delta(G) \geq 3$ (called non-path trees) separately.


\begin{lemma}
Let $T$ be a tree and $l_1$ and $l_2$ denote the number of degree one and degree two vertices, respectively. Then, any tri-connectivity augmentation set $E_{ca}$ is such that $\vert E_{ca} \vert \geq \lceil \frac{2l_1+l_2}{2} \rceil$.
\end{lemma}
\begin{proof}
It is well-known that for any $3$-connected graph $G$, $\delta(G) \geq 3$. Therefore to make $T$ a $3$-connected graph, we must increase the degree of elements in $D_1$ by at least two and the degree of elements in $D_2$ by at least one. Since an edge joins a pair of vertices, any augmentation set $E_{ca}$ has at least $\lceil \frac{2l_1+l_2}{2} \rceil$ edges. This completes the proof of the lemma. $\hfill \qed$ 
\end{proof}

\subsection{Outline of the Algorithm}
Our approach varies for the path and the non-path, for the input tree on $n$ vertices. If the input is a path, the algorithm converts the path to a cycle and then augments edges in such a way that every edge creates a cycle of length $\lfloor \frac{n}{2}\rfloor + 1$. If the input is a non-path tree, $T$: First root the tree $T$ at the maximum degree vertex, $r$. Let $\{v_1, v_2, \ldots, v_l\}$ denotes the set of leaves in $T$. Now, we group the vertex set into branches, namely $B_i$. The branch $B_i$ contains the vertices in the path from the root $r$ to the leaf $v_i$. Thus, the number of branches in the input tree is the number of degree one vertices. Next we perform a level ordering and label the degree two vertices as per the ordering as $w_1, w_2, \ldots, w_k$. Let $W = (w_1, w_2, \ldots, w_k)$. Now we initialize every vertex in $W$ as unmarked. As we iterate, we augment edges as follows: for every unmarked vertex $y \in W$ find the least unmarked vertex $x$ of different branch and if such $x$ exists for $y$, mark the vertices $x$ and $y$ and augment an edge between $x$ and $y$. Once this process is done, group the unmarked vertices in $W$. If there are no unmarked vertices then we form a cycle among the degree one vertices, if there are odd number of unmarked vertices we form a cycle among the degree one vertices and then we augment edges between the remaining degree two vertices using the ordering of vertices. If there are even number of unmarked vertices we form a path among the leaves and then we augment edges between the remaining degree two vertices using the ordering of vertices. Interestingly, this new approach guarantees that the algorithm augments exactly $\lceil \frac{2l_1+l_2}{2} \rceil$ edges.

\subsection{The Algorithm}
We now present an algorithm for tri-connectivity augmentation of trees.  Further, we show that our algorithm is optimal followed by the proof of correctness. 
  
\begin{algorithm}[H]
\caption{Tri-connectivity Augmentation of a Tree}
\label{tri}
\begin{algorithmic}[l]
\STATE{\textbf{Input}: Tree $T$}
\STATE{\textbf{Output}: Tri-vertex connected graph $H$.}
\IF{$T$ is a path}
\STATE{$Path\_Augmentation(T)$}
\ELSE
\STATE{$Non$-$path\_Augmentation(T)$}
\ENDIF
\STATE{Output $H$}
\end{algorithmic}
\end{algorithm}

\begin{algorithm}[H]
\caption{Tri-connectivity Augmentation in path like trees: \em{Path\_Augmentation(Tree $T$)}}
\label{ordering}
\begin{algorithmic}[1] 
\STATE{Let $P_{n}=(v_1, v_2, \dots, v_n)$ denotes an ordering of vertices of $T$ such that for all $1 \leq i \leq n-1$, $v_i$ is adjacent to $v_{i+1}$.}
\STATE{Augment the edge $\{v_{1},v_{n}\}$ to $T$ and update $E_{ca}$. \tt /* Converts path to a cycle */}
\FOR{$i = 1$ to $\lceil \frac{n}{2} \rceil$}
\IF{ $deg(v_i) == 2$}
\STATE{Augment the edge $\{v_i, v_{\lfloor \frac{n}{2} \rfloor + i} \}$ to $T$ and update $E_{ca}$.}
\STATE{ \tt /* Every augmented edge will create a $C_k$, $k = \lfloor \frac{n}{2} \rfloor +1 $ */}
\ENDIF
\ENDFOR
\IF{ $deg(v_n) == 2$}
\STATE{ \tt /* True, if $n$ is odd */}
\STATE{Augment the edge $\{v_n, v_{\lfloor \frac{n}{2} \rfloor + 1} \}$ to $T$ and update $E_{ca}$.}
\ENDIF
\STATE{Return the augmented graph $H$ and $E_{ca}$}
\end{algorithmic}
\end{algorithm}

\begin{algorithm}
\caption{Tri-connectivity Augmentation in non-path like Trees: \em{ Non-path\_Augmentation(Tree $T$)}}
\label{ordering1}
\begin{algorithmic}[1] 
\STATE{Let $r$ be a vertex of maximum degree and $T$ is rooted at $r$.}
\STATE{Let $D_1=\{v_1, v_2, \ldots, v_l\}$ be the set of leaves in $T$.}
\STATE{Perform Level ordering starting from $r$ and $(u_1=r, u_2, \ldots, u_n)$ denote the ordering.}
\FOR{$i=1$ to $l$}
\STATE{$B_i=V(P_{rv_i})$}
\STATE{\tt /* $B_i$ is the set of vertices in branch $i$ i.e., set of vertices in the path $P_{r,v_i}$. */}
\ENDFOR
\STATE{Perform level ordering starting from $r$ and $W = (w_1, w_2, \ldots, w_k)$, $k \leq n-l-1$, denote the ordering of degree two vertices in $T$}
\STATE{$mark[w_i]=FALSE$, $\forall ~ 1 \leq i \leq k$}
\FOR{$i=1$ to $k$}
\STATE{Find the least $j$ in $W$, $1 \leq j \leq i-1$, in such a way that $mark[w_j]=FALSE$ and there exists $s\neq t$ such that $w_i \in B_s$ and $w_j \in B_t$. }
\IF{such $j$ exists}
\STATE{Augment $\{w_i,w_j\}$ to $T$ and update $E_{ca}$}
\STATE{$mark[w_i]=mark[w_j]=TRUE$}
\ENDIF
\ENDFOR
\STATE{Let $A = \{x_1,x_2, \ldots, x_m\}$ denotes the set of unmarked vertices in $W$, where $x_i$ preserves the ordering in $w_j$}
\IF{$\vert A \vert $ = 0}
\STATE{Augment $\{v_{1},v_{l}\}$ and $\{v_{i},v_{i+1}\}$ $\forall$ $1 \leq i \leq l-1$ to $T$ and update $E_{ca}$}
\ELSIF{$ A = \{x_1, x_2, \ldots, x_m\}$, where $m$ is even}
\STATE{Augment $\{x_{2},x_{m}\}$,$\{x_{3}, x_{m-1}\}$, $\ldots$, $\{x_{\frac{m}{2}},x_{\frac{m}{2}+2}\}$ to $T$ and update $E_{ca}$}
\STATE{Augment $\{v_{i},v_{i+1}\}$ $\forall$ $1 \leq i \leq l-1$ to $T$ and update $E_{ca}$}
\IF{$\{x_{\frac{m}{2}+1},v_{l}\} \in E(T)$}
\STATE{Augment $\{x_{1},v_{l}\}$, $\{x_{\frac{m}{2}+1},v_{1}\}$ and update $E_{ca}$}
\ELSE
\STATE{Augment $\{x_{1},v_{1}\}$, $\{x_{\frac{m}{2}+1},v_{l}\}$ and update $E_{ca}$}
\ENDIF
\ELSIF{$ A = \{x_1, x_2, \ldots, x_m$\}, where $m$ is odd}
\STATE{Augment $\{x_{1},x_{m}\}$, $\{x_{2},x_{m-1}\}, \ldots, \{x_{\lfloor \frac{m}{2} \rfloor},x_{\lfloor \frac{m}{2}+2}\rfloor \}$ to $T$ and update $E_{ca}$}
\STATE{Augment $\{v_{1},v_{l}\}$ to $T$ and update $E_{ca}$}
\STATE{Augment $\{v_{i},v_{i+1}\}$ $\forall$ $1 \leq i \leq l-1$ to $T$ and update $E_{ca}$}
\IF{$\{x_{ \frac{m+1}{2} },v_{l}\} \in E(T)$}
\STATE{Augment $\{x_{ \frac{m+1}{2} },v_{1}\}$ and update $E_{ca}$}
\ELSE
\STATE{Augment $\{x_{ \frac{m+1}{2} },v_{l}\}$ and update $E_{ca}$}
\ENDIF
\ENDIF
\STATE{Return the augmented graph $H$ and $E_{ca}$}
\end{algorithmic}
\end{algorithm}

\begin{lemma}
Let T be a tree with $n \geq 4$ vertices. Algorithm $\mathtt{Path\_Augmentation()}$ yields a graph $H$, where $\forall ~ v \in V(H), deg_H(v) \geq 3$.
\end{lemma}
\begin{proof}
The algorithm, first converts the path $P_{n}=(v_1, v_2, \dots, v_n)$ to a cycle $C_{n}=(v_1, v_2, \dots, v_n)$, by adding an edge between two end vertices, i.e., the algorithm augments an edge $\{v_1,v_n\}$ in \emph{Step 2}. Now, the degree of each vertex in the resultant graph is two. In \emph{Steps 3-8}, for each vertex $v_i$ of degree two in the set $\{v_1, v_2, \ldots, v_{\lceil \frac{n}{2} \rceil}\}$, we identify a vertex $v_j$ such that the length of the path $P_{v_iv_j} = \lfloor \frac{n}{2} \rfloor + 1$ and further, we augment an edge between $v_i$ and $v_j$. Thus, in the resulting graph, degree of every vertex is three if $n$ is even and degree of all vertices other than the vertex $v_n$ is three, if $n$ is odd. So, if $n$ is odd, the algorithm augments an edge $\{v_n,v_{\lfloor \frac{n}{2} \rfloor +1}\}$ in \emph{Step 9-11}. This completes the path augmentation and in the resulting graph $H$ degree of each vertex is at least three. $\hfill \qed$
\end{proof}

\begin{lemma}
Let T be a tree with $n \geq 4$ vertices. Algorithm $\mathtt{Non}$-$\mathtt{Path\_Augmentation(T)}$ yields a graph $H$, where $\forall ~ v \in V(H), deg_H(v) \geq 3$.
\end{lemma}
\begin{proof}    
The algorithm, collects all degree two vertices and augment edges between those vertices which satisfies the condition in \emph{Step 11} and marks the end vertices of the augmented edges. Now, the marked vertices are of degree three. Collect the unmarked vertices (remaining vertices of degree 2) into the set $A$. We shall now analyze the \emph{Steps 18-37} of the algorithm by considering the following cases.

\begin{description}
\item[\textbf{Case 1:}] $ A = \emptyset$ \\
 i.e., all the degree two vertices in the given tree have become the degree three vertices in $H$. We now augment edges among the leaves such that there is a cycle $(v_1, v_2, \ldots, v_l)$. We can easily see that in the resultant graph $H$, for every vertex $v \in V(H)$,  $deg_H(v) = 3$. 

\item[\bf{Case 2:}] $A \neq \emptyset$ and $\mid A \mid$ is even, say $A = \{x_1, x_2, \ldots, x_m\}$.\\
We first form a path among leaves from $v_1$ to $v_l$ such that all the degree one vertices are converted to degree three vertices except $v_1$ and $v_l$, which is of degree two. Now, augment the edges $\{x_2,x_m\}$, $\{x_3,x_{m-1}\},$ \ldots, $\{x_{\frac{m}{2}}$,$x_{\frac{m}{2}+2}\}$. Thus, the only remaining degree two vertices are $x_1$, $x_{\frac{m}{2}+1}$, $v_1$ and $v_l$. Therefore, if $\{x_{\frac{m}{2}+1}, v_l\} \in E(T)$, then augment $\{x_1, v_l\}$ and $\{x_{\frac{m}{2}+1}, v_1\}$ and if $\{x_{\frac{m}{2}+1}, v_l\} \notin E(T)$, then augment $\{x_1, v_1\}$ and $\{x_{\frac{m}{2}+1}, v_l\}$. Hence in the resultant graph $H$ every vertex is of degree three.

\item[\bf{Case 3:}] $A \neq \emptyset$ and $\mid A \mid$ is odd, say $A = \{x_1, x_2, \ldots, x_m\}$.\\
We first form a cycle among leaves such that all the degree one vertices are converted to degree three vertices. Now, augment the edges $\{x_1,x_m\}$, $\{x_2,x_{m-1}\},$ \ldots, $\{x_{\lfloor \frac{m}{2} \rfloor}$,$x_{\lfloor \frac{m}{2}+2 \rfloor}\}$. Thus, the only remaining degree two vertex is $x_{\frac{m+1}{2}}$. Therefore, if $\{x_{\frac{m+1}{2}}, v_l\} \in E(T)$, then augment $\{x_{\frac{m+1}{2}}, v_1\}$ such that $deg_H(x_{\frac{m+1}{2}}) = 3$ and $deg_H(v_1) = 4$ and if $\{x_{\frac{m+1}{2}}, v_l\} \notin E(T)$, then augment $\{x_{\frac{m+1}{2}}, v_l\}$ such that $deg_H(x_{\frac{m+1}{2}}) = 3$ and $deg_H(v_l) = 4$. Hence in the resultant graph $H$ every vertex is of degree at least three. $\hfill \qed$
\end{description}
\end{proof}

\begin{lemma}
Let T be a tree with $n \geq 4$ vertices. Algorithm $\mathtt{Path\_Augmentation()}$ precisely augments $ \lceil \dfrac{2l_1+l_2}{2} \rceil$ edges.
\end{lemma}
\begin{proof}
 \emph{Step 2} of algorithm augments an edge between two leaves and this increases $l_2$ by two. Thus there are $l_1+l_2$ degree two vertices and \emph{Steps 3-8} augments $\lceil \frac{l_1+l_2}{2}\rceil$ new edges, if $n$ is even and $\lfloor \frac{l_1+l_2}{2}\rfloor$ new edges, if $n$ is odd. If $n$ is odd, \emph{Steps 9-12} augments an edge. Thus, if $n$ is odd, we have augmented $\frac{l_1}{2}+ \lfloor \frac{l_1+l_2}{2}\rfloor +1 = \frac{l_1}{2}+ \lceil \frac{l_1+l_2}{2}\rceil $ (Since, $l_2$ is odd). If $n$ is even, we have augmented $\frac{l_1}{2}+ \lceil \frac{l_1+l_2}{2}\rceil $ edges. In total, since $l_1 = 2$, the algorithm augments $\lceil \frac{l_1}{2}+  \frac{l_1+l_2}{2}\rceil =  \lceil \frac{2l_1+l_2}{2} \rceil $.  Therefore, the algorithm augments $\lceil \frac{2l_1+l_2}{2} \rceil$ edges in total. $\hfill \qed$
\end{proof}

\begin{lemma}
Let T be a tree with $n \geq 4$ vertices. Algorithm $\mathtt{Non}$-$\mathtt{Path\_Augmentation()}$ precisely augments $\lceil \dfrac{2l_1+l_2}{2} \rceil$ edges.
\end{lemma}
\begin{proof}
We present a proof by case analysis based on the cardinality of the set $A$ generated by \emph{Algorithm 3} in \emph{Step 17}.
\begin{description}
\item[Case 1:] $\vert A\vert =0$\\
$l_1$ edges are augmented in \emph{Step 19} by forming a cycle among leaves. In \emph{Steps 10-16}, we augment edges between the degree two vertices and since $\vert A \vert =0$, $l_2$ is even and $\lceil \frac{l_2}{2} \rceil$ edges are augmented. In total, we have augmented $l_1 + \lceil \frac{l_2}{2} \rceil$ edges,  i.e., we have augmented $\frac{2l_1}{2} +\lceil \frac{l_2}{2} \rceil = \lceil \frac{2l_1+l_2}{2} \rceil$ edges.
\item[Case 2:] $\vert A\vert = 2k, k \in \mathbf{Z}$.\\
Let $\vert A \vert = m$. Among $l_2$ degree two vertices, degree of $(l_2-m)$ vertices increases by one in \emph{Steps 10-16}, i.e., $\frac{l_2-m}{2}$ edges are augmented. Note that $\frac{m}{2}-1$ edges are augmented in \emph{Step 21}, $l_1-1$ edges are augmented in \emph{Step 22}, $2$ edges are augmented in \emph{Steps 23-27}. In total, we have augmented $\frac{l_2-m}{2} + \frac{m}{2}-1 + l_1-1+2$ edges. Since $l_2$ is even, $\frac{l_2-m}{2} + \frac{m}{2}-1 + l_1-1+2 = l_1 + \frac{l_2}{2} = l_1 + \lceil \frac{l_2}{2} \rceil = \lceil \frac{2l_1+l_2}{2} \rceil$. Therefore, $\lceil \frac{l_1+l_2}{2} \rceil$ edges are augmented.
\item[Case 3:] $\vert A\vert = 2k+1, k \in \mathbf{Z}$.\\
Let $\vert A \vert = m$. Among $l_2$ degree two vertices, degree of $(l_2-m)$ vertices increases by one in \emph{Steps 10-16} i.e., $\frac{l_2-m}{2}$ edges are augmented in \emph{Steps 10-16}. $\frac{m-1}{2} $ edges are augmented in \emph{Step 29}, $l_1$ edges are augmented in \emph{step 30-31} and a edge is augmented in \emph{Steps 32-36}. In total, we have augmented $\frac{l_2-m}{2} +  \frac{m-1}{2}  + l_1 +1$ edges. Since $m$ is odd and $(l_2-m)$ is even, $\frac{l_2-m}{2} +  (\frac{m-1}{2} +1 )  + l_1  = \lceil \frac{l_2-m}{2} \rceil + \lceil \frac{m}{2} \rceil + l_1 =  \lceil \frac{2l_1+l_2}{2} \rceil$ edges are augmented in total. 
\end{description} 
Thus, the algorithm augments $\lceil \frac{2l_1+l_2}{2} \rceil$ edges. $\hfill \qed$
\end{proof}

\begin{lemma}
For a tree $T$, the graph obtained from the algorithm $\mathtt{Path\_Augmentation()}$ is 3-connected.
\end{lemma}
\begin{proof}
Our claim is to prove that every minimal vertex separator is of size at least $3$ and there exist at least one vertex of degree $3$. On the contrary, assume that there exist at least one minimal vertex separator of size at most 2, say $\vert S\vert \leq 2$. Let $C = (v_1, v_2, \ldots, v_n)$ be the cycle formed in \emph{Step 2} of \emph{Algorithm 2}, where $\{v_1, v_n\} \in E_{ca}$.
\begin{description}
\item[Case 1:] $\vert S \vert =1$. \\
 For every vertex $v_i \in V(H)$, $v_i \in V(C)$. Hence, $H\backslash S$ is connected.
\item[Case 2:] $\vert S \vert =2$. Let $S = \{v_i,v_j\}, i \neq j$ and $1 \leq i,j \leq n$. 
\begin{description}
\item[Case 2.1:] $\{v_i,v_j\} \in E(C)$. Since $C$ is a cycle, $H\backslash S$ is connected.
\item[Case 2.2:] $\{v_i,v_j\} \in E_{ca}\backslash E(C)$ or $\{v_i,v_j\} \notin E(H)$\\
For every internal vertex $v_k$ in the path $P'_{v_iv_j} = \{v_i, v_{i+1}, \ldots, v_j\}$ there exist a vertex $v_l \in V(H)$ such that $\{v_k, v_l\} \in E(H)$ and $v_l \notin V(P'_{v_iv_j})$ by \emph{Steps 3-12} of \emph{Algorithm 2}. Thus, $H\backslash S$ is connected.
\end{description}
\end{description}
In all the above cases, the graph $H\backslash S$ is connected, which is a contradiction to the assumption that $S$ is a vertex separator. Therefore, every minimal vertex separator of $H$ is of size at least 3. Note that by our augmentation procedure, $deg_H(v_1) = 3$.  Clearly, $N_H(v_1)$ is a minimal vertex separator of size three. Thus, the graph $H$ is 3-connected. $\hfill \qed$
\end{proof}

\begin{lemma}
For a tree $T$, the graph obtained from the algorithm $\mathtt{Non}$-$\mathtt{Path\_Augmentation()}$ is 3-connected.
\end{lemma}
\begin{proof}
It is enough to prove that the size of every minimal vertex separator is at least 3 and there exists at least one minimal vertex separator of size 3. On the contrary, assume that there exist at least one minimal vertex separator $S$ such that $\vert S \vert \leq 2$. Let $P = \{v_1, v_2, \ldots, v_l\}$ be the path formed in \emph{Steps 18-37} of \emph{Algorithm 3}, where $E(P) \subset E_{ca}$.
\begin{description}
\item[Case 1:] $\vert S \vert = 1$. Let $S = \{u\}$. The vertex $u$ can be a root node, $r$, or a node in the path, $P$, or neither. By case analysis, we prove that the graph $H\backslash S$ is connected, which forms a contradiction to the definition of $S$. \\
\begin{description}
\item[Case 1.1:] If $u=r$, then since $P$ is the path $H \backslash S$ is connected. 
\item[Case 1.2:] If $u \in V(P)$, then, since $deg_T(u) = 1$, $T \backslash S$ is connected, Hence $H \backslash S$ is also connected. 
\item[Case 1.3:] If $u \in V(H)\backslash (V(P)\cup \{r\})$, then every vertex $w \in P_{ru}\backslash \{u\}$ has a path $P_{wr}$ such that $u \notin V(P_{wr})$ and every vertex $x \in P_{uv_i}\backslash \{u\}$, where $v_i \in D_1$ for some $1\leq i\leq l$ such that $P_{uv_i}$ exist in $T$, has a path $P_{xv_i}$ such that $u \notin V(P_{xv_i})$. Thus, the graph $H\backslash S$ is connected.
\end{description}
\item[Case 2:] If $\vert S \vert = 2$, say $S=\{u,v\}$. $S$ can either be a clique or an independent set. If $S$ is a clique then either the edge is from the tree $T$ or from the augmentation set $E_{ca}$. In each case, we prove that the graph $H\backslash S$ is connected, which is a contradiction to the definition of $S$.

\begin{description}
\item[Case 2.1:] $\{u,v\} \in E(T)$. 

\begin{itemize}
\item[$\bullet$] If either $u$ or $v$ is a root node. Without loss of generality, let $u$ be the root node. By the path $P$ in $H$, every pair of vertex $w,x \in V(H)\backslash S$ has a path connecting them in $H\backslash S$. Thus, the graph $H\backslash S$ is connected. 
\item[$\bullet$] If neither $u$ nor $v$ is a root node. In $H\backslash S$, for every vertex $w \in V(P_{ru})\backslash \{u\}$, there exists a path $P_{wr}$ such that $u,v \notin V(P_{wr})$ and for every vertex $x \in V(P_{vv_i})$, where $v_i \in D_1$ for some $1\leq i\leq l$ such that $P_{uv_i}$ exist in $T$, there exists a path from $P_{xv_i}$ such that $u,v \notin V(P_{xv_i})$. Since $deg_H(r) \geq 3$ and by the path $P$ in $H$, the graph $H\backslash S$ is connected.
\end{itemize} 

\item[Case 2.2:] $\{u,v\} \in E_{ca}$.

\begin{itemize}
\item[$\bullet$] If $deg_T(u)=1$ and $deg_T(v)=1$ then, the graph $T\backslash S$ is connected. Thus, the graph $H\backslash S$ is connected.

\item[$\bullet$] If either $deg_T(u)=1$ or $deg_T(v)=1$. $w.l.o.g$, assume that $deg_T(v)=1$. 
\begin{itemize}
\item[-] If $u,v \in B_i,$ for some $1 \leq i \leq l$. \\ 
Since $\forall ~ w \in V(H), deg_H(w) \geq 3$, every internal vertex of $P_{uv}$ in $T$ contributes degree 2 to the path $P_{uv}$ and the remaining degree to the vertices which does not belong to $V(P_{uv})$ (by \emph{Steps 10-37}). For every internal vertex $w \in V(P_{ru})$ of $T$, there exists a path from $P_{wr}$ such that $u,v \notin V(P_{wr})$. Since $deg_H(r) \geq 3$ and by the path $P$, the graph $H\backslash S$ is connected.

\item[-] If $u \in B_i$ and $v \in B_j$ for some $1 \leq i,j \leq l$ and $i \neq j$.\\
Since $deg_T(v)=1$, $H\backslash \{v\}$ is connected. For every internal vertex $w \in V(P_{ru})$ of $T$, there exists a path from $P_{wr}$ such that $u,v \notin V(P_{wr})$ and for every internal vertex $x \in V(P_{uv_i})$ of $T$, where $v_i \in D_1$ for some $1\leq i\leq l$ such that $P_{uv_i}$ exist in $T$, there exists a path from $P_{xv_i}$. Since $deg_H(r) \geq 3$ and by the path $P$, the graph $H\backslash S$ is connected.
\end{itemize}

\item[$\bullet$] If neither $deg_T(u)= 1$ nor $deg_T(v) = 1$. 
\begin{itemize}
\item[-] If $u,v \in B_i,$ for some $1 \leq i \leq l$.\\
 i.e., the edge $\{u,v\}$ is augmented by the \emph{Step 21} or the \emph{Step 29}. Thus, there exists an internal vertex $y \in V(P_{uv})$ in $T$ and one of the internal vertex in the path $P_{uv}$ is adjacent to a vertex in different branch, by the \emph{Steps 23-27} or by the \emph{Steps 32-36}. For every internal vertex $w \in V(P_{ru})$ of $T$, there exists a path $P_{wr}$ such that $u,v \notin V(P_{wr})$, since if $d_T(w) = 2$ then the vertex $w$ is augmented to a vertex in different branch by the least degree two vertex condition in \emph{Step 11} and if $d_T(w) \geq 3$ then there exists a path from $w$ to a vertex in different branch in $T$. For every internal vertex $x \in V(P_{uv_i})$ of $T$, where $v_i \in D_1$ for some $1\leq i\leq l$ such that $P_{uv_i}$ exist in $T$, there exists a path $P_{xv_i}$. Since $deg_H(r) \geq 3$ and by the path $P$, the graph $H\backslash S$ is connected.
 
\item[-] If $u \in B_i$ and $v \in B_j$ for some $1 \leq i,j \leq l$ and $i \neq j$.\\
i.e., the edge $\{u,v\}$ is augmented by the \emph{Steps 10-16}. For every internal vertex $w \in V(P_{ru})$ ($a \in V(P_{rv})$ ), there exists a path $P_{wr}$ ($P_{ar}$) such that $u,v \notin V(P_{wr})$ ($u,v \notin V(P_{ar})$) and for every internal vertex $x \in V(P_{uv_i})$ ($b \in V(P_{vv_i})$), where $v_i \in D_1$ for some $1\leq i\leq l$ such that $P_{uv_i}$ exist in $T$, there exists a path $P_{xv_i}$ ($P_{bv_i}$). Since $deg_H(r) \geq 3$ and by the path $P$, the graph $H\backslash S$ is connected.

\end{itemize}
\end{itemize} 

\item[Case 2.3:] $\{u,v\} \notin E(H)$.
\begin{itemize}
\item[$\bullet$] If $u=r$ or $v=r$
Without loss of generality, assume that $u=r$. Then $deg(u) = k$, $k \geq 3$. Since, all the degree one vertices are connected by a path $P$, the graph $H\backslash \{u\}$ is connected. Every internal vertex $w\in P_{uv}$ of $T$ contributes degree 2 to the path $P_{uv}$ and remaining to the vertices which are not in the path $P_{uv}$. Thus, the graph $H\backslash S$ is connected.
\item[$\bullet$] If $u,v \in B_i,$ for some $1 \leq i \leq l$.
Similar argument as in \emph{Case 2.2}.
\item[$\bullet$] If $u \in B_i$ and $v \in B_j$ for some $1 \leq i,j \leq l$ and $i \neq j$.
Similar argument as in \emph{Case 2.2}.
\end{itemize}
\end{description}
\end{description}
In all the above cases, the graph $H\backslash S$ is connected, which is a contradiction to the fact that $S$ is a vertex separator. Therefore, our assumption that there exists at least one minimal vertex separator in $H$ of size at most two is wrong. Hence, every minimal vertex separator of $H$ is of size at least 3.\\
Now, our claim is to prove that there exists at least one minimal vertex separator of size 3. By \emph{Lemma 2}, the degree of every vertex in $V(P)\backslash \{v_l\}$ is 3. The graph $H$ is 3-connected. $\hfill \qed$
\end{proof}

\begin{theorem}
For a tree $T$, the graph $H$ obtained from \emph{Algorithm 1} is 3-connected. Further, $H$ is obtained from $T$ by augmenting a minimum set of edges.
\end{theorem}
\begin{proof}
The lower bound for the tri-connectivity augmentation of trees is $\lceil \frac{2l_1+l_2}{2} \rceil$ by \emph{Lemma 1}. If the tree $T$ is a path, \emph{Algorithm 1} calls \emph{Algorithm 2}, which converts the tree to a 3-connected graph $H$ by augmenting exactly $\lceil \frac{2l_1+l_2}{2} \rceil$ edges (by \emph{Lemma 4} and \emph{Lemma 6}). If the tree $T$ is a non-path tree, the \emph{Algorithm 1} calls \emph{Algorithm 3}, which converts the tree to a 3-connected graph $H$ by augmenting exactly $\lceil \frac{2l_1+l_2}{2} \rceil$ edges (by \emph{Lemma 5} and \emph{Lemma 7}). Thus, for a tree $T$ the graph obtained from \emph{Algorithm 1} is 3-connected. Further, $H$ is obtained by using a minimum connectivity augmentation set. Therefore, the claim follows. $\hfill \qed$
\end{proof}

\subsection{Implementation and Analysis of the algorithm}
Let $T$ be a tree with the vertex set $V(T)$ such that $\vert V(T) \vert =n$, with the edge set $E(T)$ such that $\vert E(T) \vert =m$ and $l$ be the number of leaves. The Algorithm $\mathtt{Path\_Augmentation()}$ takes $O(1)$ time in \emph{Step 2}, $O(n)$ time for \emph{Steps 3-8} and $O(1)$ time for \emph{Steps 9-12} of \emph{Algorithm 2}. Thus, the Algorithm $\mathtt{Path\_Augmentation()}$ takes $O(n)$ time. 

 The Algorithm $\mathtt{Non}$-$\mathtt{Path\_Augmentation()}$: Since, the level ordering can be implemented by the data structure $QUEUE$, the \emph{Step 3} takes $O(n)$ time, takes $O(ln)$ time for \emph{Steps 4-7} of \emph{Algorithm 3} and $O(n)$ time for the \emph{Step 8}. We implement the data structure QUEUE for the \emph{Steps 8-16}, which is used to keep track of marked and unmarked vertices in $W$. This process ends after visiting all vertices in $T$ and it takes $O(n^2)$ time for \emph{Steps 10-16} and the algorithm takes $O(n)$ time for \emph{Steps 18-38}. In total, the Algorithm $\mathtt{Non}$-$\mathtt{Path\_Augmentation()}$ takes $O(n^2)$ time. Therefore, \emph{Algorithm 1} takes $O(n^2)$ time. Thus, for a given tree, a minimum tri-connectivity augmentation set can be found in $O(n^2)$ time.

\subsection{Trace of the Algorithm (Algorithm 2)}

\begin{figure}[H]
\centering
\includegraphics[scale=0.40]{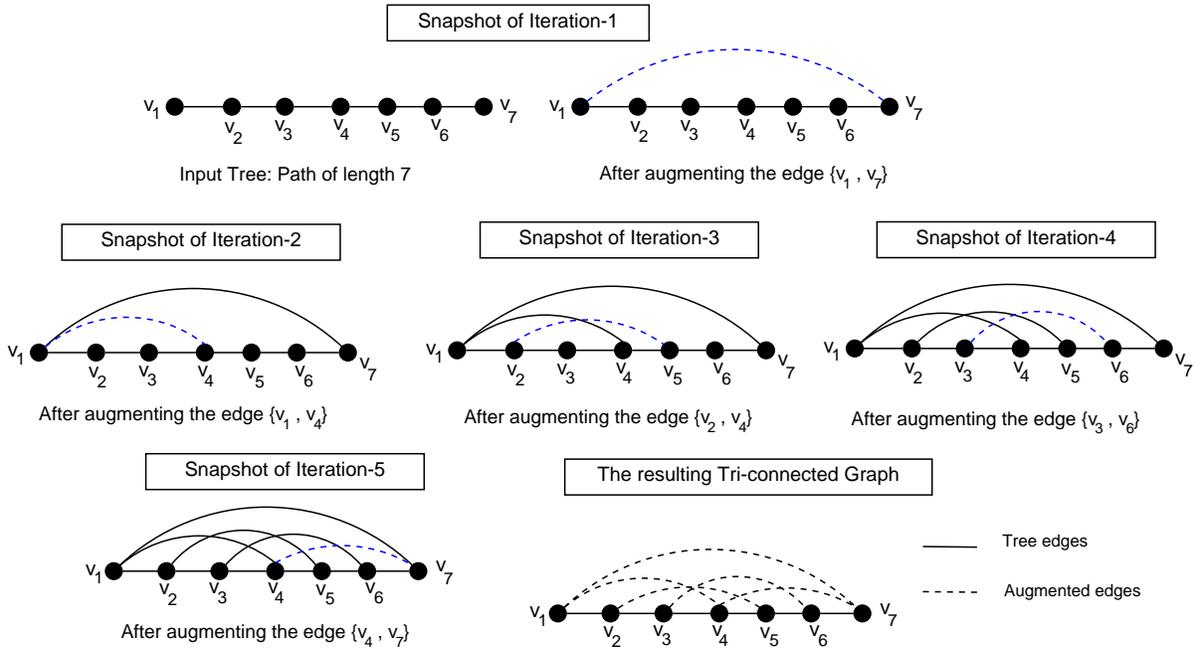}
\caption{Trace of Algorithm 2 when $n = 7$ }
\end{figure}

\subsection{Trace of the Algorithm (Algorithm 3)}

\begin{figure}[H]
\centering
\includegraphics[scale=0.35]{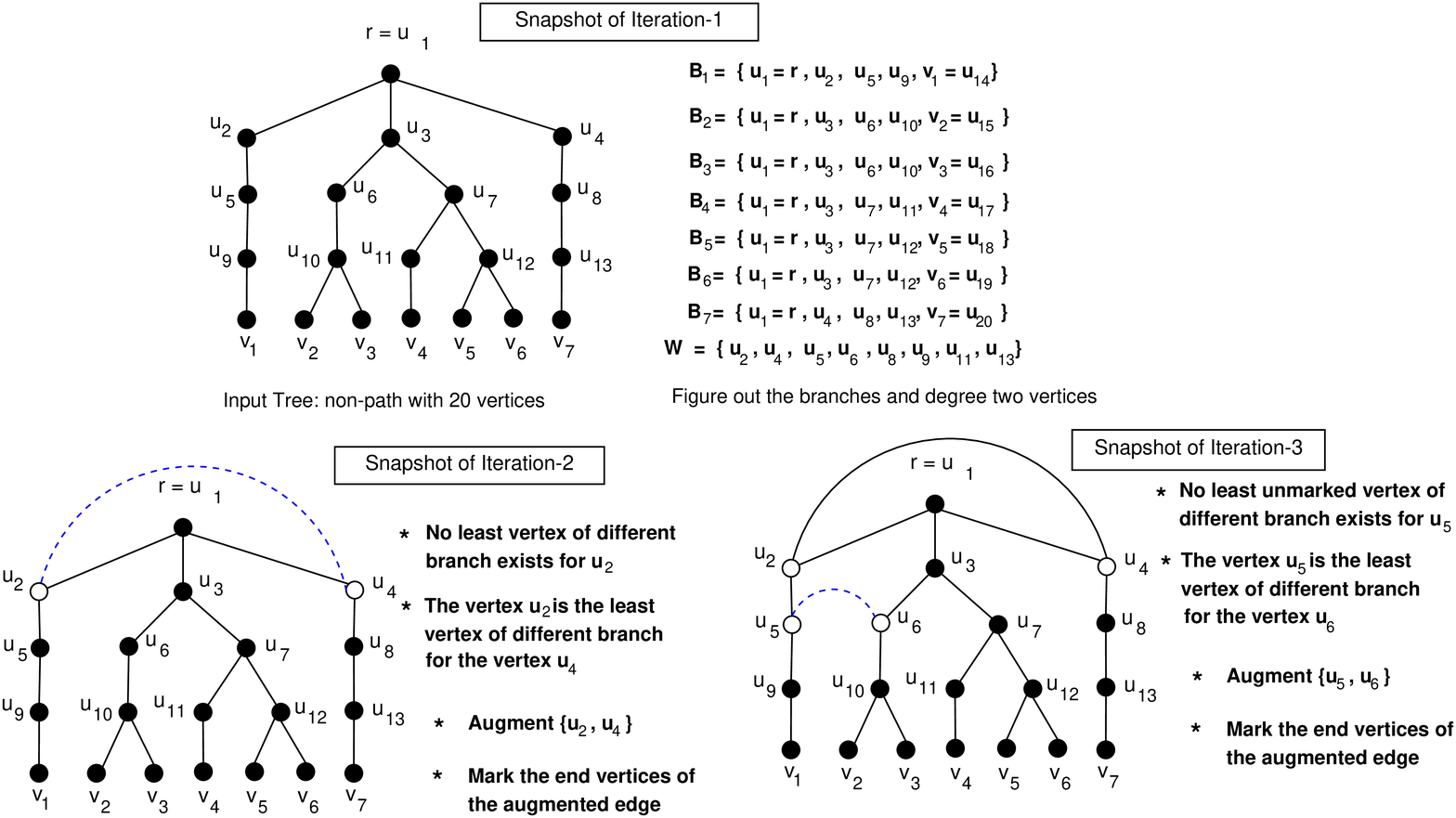}\\
\includegraphics[scale=0.37]{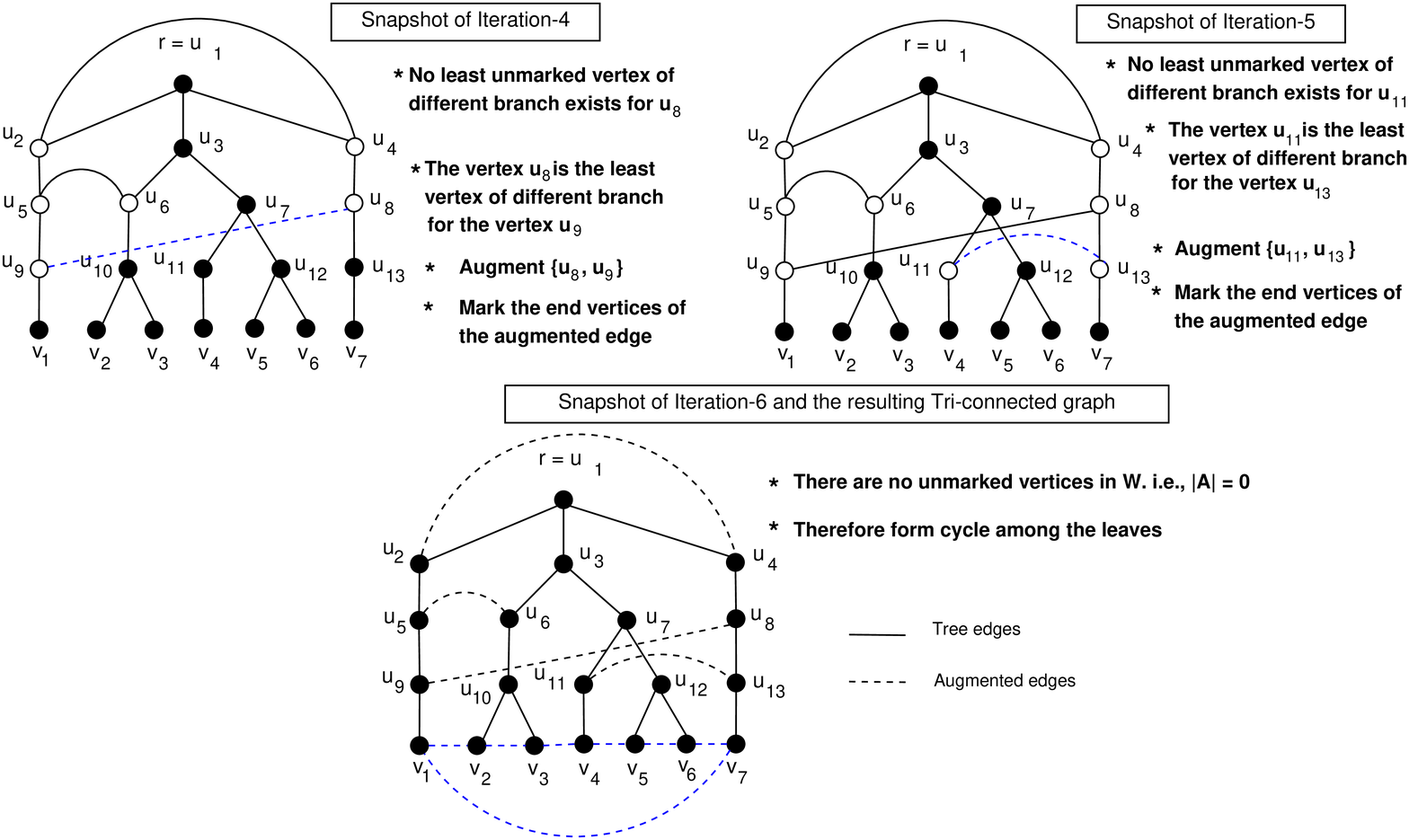}
\caption{An example for the tri-connectivity augmentation of a non-path tree, which satisfies the condition $\vert A \vert = 0$}
\end{figure}

\section{Conclusions and Future Directions}
In this paper, we have presented an algorithm for finding a minimum tri-connectivity augmentation set in trees. We believe that the approach can be extended to $1$-connected graph  with the help of block trees, biconnected component trees proposed in \cite{tarjan,nsn}. A logical extension of this work would be to look at $r$-connectivity augmentation of trees for any $r \geq 2$.

\end{document}